\documentclass{amsart}
\usepackage{amsmath}
\usepackage{amsfonts}
\usepackage{amstext}
\usepackage{amsbsy}
\usepackage{amsopn}
\usepackage{amsxtra}
\usepackage{upref}
\usepackage{amsthm}
\usepackage{amsmath}
\usepackage{amssymb}

\usepackage[textures,bookmarks=false]{hyperref}
\usepackage{mathrsfs}
\usepackage{enumerate}

\newtheorem{prop}{Proposition}[section]
\newtheorem{rem}{Remark}[section]
\newtheorem{lema}{Lemma}[section]
\newtheorem{defi}{Definition}[section]
\newtheorem{teo}{Theorem}[section]
\newtheorem{maintheorem}{Theorem}
\newtheorem{eje}{Example}[section]
\newtheorem{coro}{Corollary}[section]

\newtheorem{claim}{Claim}
\newtheorem*{claim*}{Claim}

\def\eps{\varepsilon}
\def\phi{\varphi}
\def\R{{\mathbb R}}

\def\N{{\mathbb N}}

\def\P{{\mathcal P}}

\def\F{{\mathcal F}}

\def\M{{\mathcal M}}

\def\es{{\emptyset}}
\def\sm{\setminus}

\def\crit{{\mathcal Cr}}

\def\bd{\partial }
\def\le{\leqslant}
\def\ge{\geqslant}
\def\st{such that }

\def\F{\mathcal{F}}
\def\M{\mathcal{M}}
\def\dimspec{\mathfrak{D}}
\def\htop{h_{top}}
\def\Sm{ \medspace\Big\backslash}

\title[Dimension theory for multimodal maps]{Dimension theory for multimodal maps}
\date{\today}

\begin{thanks}
{ GI was partially  supported by  Proyecto Fondecyt 11070050. MT was partially supported by FCT grant SFRH/BPD/26521/2006 and also by FCT through CMUP}
\end{thanks}



\author{Godofredo Iommi} \address{Facultad de Matem\'aticas,
Pontificia Universidad Cat\'olica de Chile (PUC), Avenida Vicu\~na Mackenna 4860, Santiago, Chile}
\email{giommi@mat.puc.cl}
\urladdr{http://www.mat.puc.cl/\textasciitilde giommi/}
\author{Mike Todd} \address{Centro de Matem\'atica da Universidade do Porto, Rua do Campo Alegre 687, 4169-007 Porto, Portugal \footnote{
{\bf Current address:}\\
Department of Mathematics and Statistics\\
Boston University\\
111 Cummington Street\\
Boston, MA 02215\\
USA }\\ }
\email{mtodd@math.bu.edu}
\urladdr{http://math.bu.edu/people/mtodd/}

\begin{document}

\begin{abstract}
This paper is devoted to the study of dimension theory, in particular  multifractal analysis, for multimodal maps. We describe the Lyapunov spectrum, generalising previous results by Todd. We also study the multifractal spectrum of pointwise dimension.  The lack of regularity of the thermodynamic formalism for this class of maps is reflected in the phase transitions of the spectra.
\end{abstract}

\maketitle
\section{Introduction}

The dimension theory of dynamical systems has flourished remarkably over the last fifteen years. The main goal of the field is to compute the \emph{size} of dynamically relevant subsets of the phase space.  For example, sets where the complicated dynamics is concentrated (\emph{repellers} or \emph{attractors}). Usually, the geometry of these sets is rather complicated. That is why there are several notions of \emph{size} that can be used. One could say that a set is \emph{large} if it contains a great deal of disorder on it. Formally, one would say that the dynamical system restricted to that subset has large \emph{entropy}. Another way of measuring the size of a set is using geometrical tools, namely Hausdorff dimension. There are  usually two conditions required on the dynamical system $(X,f)$ for the dimension theory to be carried out.  Firstly, a certain amount of hyperbolicity enables us to use Markov partitions and the thermodynamic formalism machinery associated to the Markov setting.  Secondly, the geometrical nature of Hausdorff dimension means that it is convenient to assume that the map $f$ is conformal. In this case, the elements of a Markov partition are \emph{almost} balls, hence can be used to compute Hausdorff dimension (see \cite{Bar,Pesbook}  and reference therein).

 In this paper we consider smooth one dimensional maps. This implies that the map is conformal. Nevertheless, we study dynamical systems for which the hyperbolicity is rather weak (these maps have critical points).
The class of maps we will consider is defined as follows. We say that an interval map $f$ is $C^{1+}$ if it is $C^1$ and the derivative $Df$ is $\alpha$-H\"older for some $\alpha>0$.  Let $\mathcal F$ be the collection of $C^{1+}$ multimodal interval maps $f:I \to I$, where $I=[0,1]$, satisfying:

\newcounter{Lcount}
\begin{list}{\alph{Lcount})}
{\usecounter{Lcount} \itemsep 1.0mm \topsep 0.0mm \leftmargin=7mm}

\item the critical set $\crit = \crit(f)$ consists of finitely many critical points $c$ with critical order $1 < \ell_c < \infty$, i.e., there exists a neighbourhood $U_c$ of $c$ and a diffeomorphism $g_c:U_c \to g_c(U_c)$ with $g_c(c) = 0$
     $f(x) = f(c) + g_c(x)^{\ell_c}$;
\item $f$ has negative Schwarzian derivative, i.e., $1/\sqrt{|Df|}$ is convex; 
\item $f$ is topologically transitive on $I$;
\item $f^n(\crit)  \cap f^m(\crit)=\es$ for $m \neq n$.
\end{list}

Note that maps in this class do not have parabolic fixed points since the negative Schwarzian condition rules out repelling parabolic cycles and topological transitivity rules out attracting parabolic cycles. Indeed, this follows from Singer's Theorem \cite[Theorem II.6.1]{MSbook} plus the negative Schwarzian condition (the $C^3$ condition can be removed by for example \cite{Ced}).  We refer to \cite[Remarks 1.1 and 1.2]{IT} for more information on this choice of family of maps.
The thermodynamic formalism for these maps was studied in \cite{IT}.  We proved that in an interval of the form $(-\infty, t^+)$ for some $t^+>0$, the pressure function $t \to P(-t \log |Df|)$ is strictly concave, $C^1$ and the  `natural/geometric' potential $x\mapsto - t \log |Df(x)|$ has a unique
equilibrium state (see Section \ref{sec:therm form} for precise definitions and statements). In particular in the interval $(-\infty, t^+)$ the thermodynamic formalism has similar properties as in the uniformly hyperbolic case. In the interval $(t^+, \infty)$ the pressure function is linear. In particular, at the point  $t=t^+$ it exhibits a so-called \emph{first order phase transition}, that is a point where it is not smooth. This lack of regularity is closely related to the different modes of recurrence of the system (see \cite{Sartherm, Sarphase}).

If $f$ has an absolutely continuous invariant probability measure (acip) $\mu$, we will often denote this measure by $\mu_{ac}$.  We define $\F_{ac}\subset \F$ to be the class of maps with an acip $\mu_{ac}$ and with $p>1$ such that $\frac{d\mu_{ac}}{dm}\in L^p(m)$ where $m$ denotes Lebesgue measure.

This paper is devoted to the study of the dimension theory for maps in  $\F_{ac}$. In particular, we are interested in its  multifractal analysis (see Section \ref{sec:prelim dim} for precise definitions).
Our first goal is to study the \emph{Lyapunov spectrum} (see Section \ref{sec:prelim dim} for precise definitions). Making use of the thermodynamic formalism we are able to describe the size (Hausdorff dimension) of the level sets determined by the Lyapunov exponent of these maps. We denote by $J(\lambda)$ the set of points having $\lambda$ as Lyapunov exponent. Dynamical and geometrical features are captured in this decomposition. We next state our first main theorem: the set $A$ is defined in Section~\ref{sec:Lyap spec}.

\begin{maintheorem} \label{thm:main Lyap}
 Suppose that $f\in \F_{ac}$.   Then for all $\lambda \in \R\sm A$,
\begin{equation*}
L(\lambda):= \dim_H(J(\lambda)) = \frac{1}{\lambda} \inf_{t \in \mathbb{R}}\left(P(-t \log |Df|) + t \lambda\right).
\end{equation*}
\end{maintheorem}
In particular, the function $\lambda \mapsto \lambda \dim_H(J(\lambda)) $ is the Legendre-Fenchel transform of the pressure function $t \mapsto P(-t \log |Df|)$. The lack of hyperbolicity of $f \in \F_{ac}$ is reflected in the lack of regularity of the pressure functions (\emph{phase transitions}). Therefore, the  Lyapunov spectrum keeps track of all the changes in the recurrence modes of the system (see \cite{Sarphase}).

Theorem~\ref{thm:main Lyap} has been proved in different settings with different assumptions on the hyperbolicity of the system. For example, for the Gauss map it was proved by Kesseb\"ohmer  and  Stratmann \cite{KeS}; for maps with parabolic fixed points related results were shown in \cite{BarIo,GelRa,JoRa,KeS,Nak,PolWe}; and for maps with countably many branches and parabolic fixed points this was shown by Iommi \cite{Iom}.   For rational maps on the complex plane a similar result was recently shown by Gelfert,  Przytycki and Rams \cite{GelPrRa}.

We also study the multifractal spectrum of the pointwise dimension of equilibrium measures for H\"older potentials. The first thing that needs to be proved is that in this non-uniformly hyperbolic setting, H\"older potentials have unique equilibrium states.  We first need to define the class of maps we can apply our results to: $\F_D\subset \F$  is the class of $C^3$ maps with $|Df^n(f(c))|\to \infty$ for any $c\in \crit$.  The main reason we impose these conditions is to ensure that `induced potentials' we deal with in this paper are sufficiently smooth, an application of \cite[Lemma 4]{BTeqgen}.  Moreover, as in \cite{BRSS}, $\F_D\subset \F_{ac}$.  The increased smoothness of the map allows us to study the pointwise dimension for equilibrium states for $\phi$ through the study of potentials of the form $-t\log|Df|+s\phi$.  See \cite{Pesbook} for a general account of this approach.  So for example, as shown in Section~\ref{sec:temp}, using \cite{IT}, we obtain:

\begin{teo}\label{thm:eq cts pots}
Suppose $f\in \F_D$ and $\phi: I \to \R$ is a H\"older potential  with $\phi<P(\phi)$.  Then there exists $\eps>0$ such that for each $t\in (-\eps, \eps)$ there is a unique equilibrium state $\mu_{\phi+t\log|Df|}$ for $(I, f,\phi+t\log|Df|)$.  Moreover $h(\mu_{\phi+t\log|Df|})>0$.
\end{teo}

This theorem was proved by Bruin and Todd \cite{BTeqgen} for a narrower range of potentials $\phi$: potentials not too far from the constant function.  So in some ways the above theorem is an improvement on Bruin and Todd's results.  However, we note that the statistical properties of the equilibrium states in \cite{BTeqgen} and the relevant properties of the pressure function are stronger.

\begin{rem}\label{rmk:conf}
Existence of a conformal measure for the H\"older potential $\phi$ as in Theorem~\ref{thm:eq cts pots} follows from \cite{DenUrb} and uniqueness follows as in \cite[Theorem 8]{Dob_comp} and \cite[Theorem 35]{Dobcusp}.  For a discussion of the different classes of smoothness of potentials required to guarantee the existence of equilibrium states see \cite[Section 1]{BTeqgen}.
\end{rem}

We describe the multifractal decomposition induced by the pointwise dimension of equilibrium measures for H\"older potentials (see Section \ref{sec:prelim dim} for precise definitions). When considering uniformly  hyperbolic dynamical systems and H\"older potentials the multifractal spectrum of pointwise dimension is very regular, indeed it has bounded domain, it is strictly concave and real analytic (see \cite[Chapter 7]{Pesbook}).  In our setting the multifractal spectrum can exhibit different behaviour. Not only can it have unbounded domain, but it can also have points where it is not analytic and sub-intervals where it is not strictly concave. This is a consequence of the lack of hyperbolicity of our dynamical systems and of the  following result (see Section \ref{sec:dspec proof}).
Given $(I,f, \phi)$ we let $\mu_{\phi}$ be the equilibrium state and denote
 \[T_\phi(q):= \inf \{ t \in \mathbb{R} : P(-t \log |Df| +q \phi) =0\}. \]
We show in Section~\ref{sec:temp} that this function is $C^1$ and strictly convex in an interval $(q_\phi^-, q_\phi^+)$.  The size of this interval is discussed in Section~\ref{sec:temp} and in Remark~\ref{rem:lack of dim upper bd}.  We let $\dimspec_{\mu_\phi}(\alpha)$ be the multifractal spectrum of pointwise dimension 
(see Section~\ref{sec:dim spec intro} for definitions). Note that in \cite[Theorem A]{T} it was necessary to restrict the results to points with positive pointwise upper Lyapunov exponent.  In the next result we are able to remove this restriction by results of \cite{rivshen}. For further discussion of the condition that $\mu_\phi$ is not absolutely continuous with respect to the Lebesgue measure see Section~\ref{sec:dspec proof}.

 \begin{maintheorem} \label{thm:main multi}
   Suppose that $f\in \F_D$ and $\phi:I \to \R$ is a H\"older potential with $\phi<P(\phi)=0$.  If $\mu_\phi$ is not absolutely continuous with respect to the Lebesgue measure then the dimension spectrum satisfies
\begin{equation*}
 \dimspec_{\mu_\phi}(\alpha)= \inf_{q \in \mathbb{R}} \left(T_\phi(q) +q \alpha \right)
\end{equation*}
for all $\alpha \in (-DT_\phi(q_\phi^+), -D^+T_\phi(q_\phi^-))$.
\end{maintheorem}

This formula for the dimension spectrum was first rigorously  proved by   Olsen \cite{Ols}  and by Pesin and Weiss \cite{PesWei_mult} for uniformly hyperbolic maps and for Gibbs measures.  The case of the Manneville Pomeau map (non-uniformly hyperbolic map) was studied by Nakaishi \cite{Nak},  by Pollicott and Wiess \cite{PolWe} and by Jordan and Rams \cite{JoRa}.
The case of Horseshoes with a parabolic fixed point was considered in Barreira and Iommi \cite{BarIo}. Multifractal analysis of pointwise dimension was also considered in  the countable Markov shift setting by Hanus, Mauldin and  Urba\'nski and \cite{HaMauUr} and by Iommi \cite{io1}.   For general piecewise continuous maps, analysis of this type was addressed in \cite{HoRaSt}.  For multimodal maps the multifractal analysis of pointwise dimension study began with the work of Todd \cite{T}.

As in \cite{T}, the main tool we use to prove our main results is a family of so-called \emph{inducing schemes}
(see Section~\ref{subsec:ind_sch}). These are dynamical systems associated to $f$ which on the one hand have better expansion and hyperbolicity properties but on the other are defined on a non-compact space.  We translate our problems to this setting, solve it there and then push the results back into the original system.  We use the fact that $f\in \F_{ac}$ to ensure that this process does not miss too many points.

\emph{The structure of the paper.} In Section~\ref{sec:prelim dim} we define the notions we will use from dimension theory.  In Section~\ref{sec:prelim erg} we define the ideas we need from thermodynamic formalism, introduce our inducing schemes and then discuss thermodynamic formalism for inducing schemes. In Section~\ref{sec:Lyap spec} we prove Theorem~\ref{thm:main Lyap}.  We give some basic ideas for the dimension spectrum in Section~\ref{sec:dim spec intro}.  We  set up the proof of Theorem~\ref{thm:main multi} in Section~\ref{sec:temp} and then prove the theorem in Section~\ref{sec:dspec proof}.

\emph{Acknowledgements.}  MT would like to thank the maths department of Pontificia Universidad Cat\'olica de Chile, where some of this work was carried out, for their hospitality.  Both authors would like to thank J. Rivera-Letelier for useful conversations and to Henk Briun and Neil Dobbs for useful remarks.

\section{Preliminaries: dimension theory}\label{sec:prelim dim}
Here we recall  basic definitions and results from dimension theory (see \cite{Pesbook} and \cite{PrzUrb_book} for details). A countable collection of sets $\{U_i \}_{i\in N}$ is called a $\delta$-cover of $F \subset\R$ if $F\subset\bigcup_{i\in\N} U_i$, and $U_i$ has diameter $|U_i|$ at most $\delta$ for every $i\in\N$. Let $s>0$, we define
\[
\mathcal{H}^s(F) := \lim_{\delta \to 0}\inf \left\{ \sum_{i=1}^{\infty} |U_i|^s : \{U_i \}_i \text{ a } \delta\text{-cover of } F \right\}.
\]
The \emph{Hausdorff dimension} of the set $F$ is defined by
\[
{\dim_H}(F) := \inf \left\{ s>0 : \mathcal{H}^s(F) =0 \right\}.
\]
Given a finite Borel measure $\mu$ in $F$, the \emph{pointwise dimension} of $\mu$ at the point $x$ is defined by
\[
d_{\mu}(x) := \lim_{r \to 0} \frac{\log \mu (B(x,r))}{\log r},
\]
whenever the limit exists, where $B(x,r)$ is the ball at $x$ of radius~$r$.  This function describes the power law behaviour of
$\mu(B(x,r))$ as $r \rightarrow 0$, that is
\[\mu(B(x,r)) \sim r^{d_{\mu}(x)}.\]
The pointwise dimension quantifies how concentrated a measure is around a point: the larger it is  the less concentrated the measure is around that point.
Note that if $\mu$ is an atomic measure supported at the point $x_{0}$ then $d_{\mu}(x_{0})=0$ and if $x_{1} \neq x_{0}$ then $d_{\mu}(x_{1}) = \infty $.

The following propositions relating the pointwise dimension with the Hausdorff dimension can be found in \cite[Chapter 2 p.42]{Pesbook}.

\begin{prop}\label{prop:upper}
Given a finite Borel measure $\mu$, if $d_{\mu}(x) \le d$ for every $x \in F$, then ${\dim_H}(F)\le d$.
\end{prop}

The \emph{Hausdorff dimension} of the measure $\mu$ is defined by
\[
{\dim_H}(\mu) := \inf \left\{ {\dim_H}(Z): \mu(Z)=1 \right\}.
\]

\begin{prop}\label{prop:lower}
Given a finite Borel measure $\mu$, if $d_{\mu}(x) =d$ for $\mu$-almost every $x \in F$, then ${\dim_H}(\mu)=d$.
\end{prop}

In this paper we will be interested in several types of multifractal spectra. In order to give a unified definition of the objects and of the problem we will present the general concept of multifractal analysis as developed by Barreira, Pesin and Schmeling \cite{BarPeSc} (see also \cite[Chapter 7]{Bar}).

Consider a function $g:Y \to [-\infty, +\infty]$, where $Y$ is a subset of the space $X$.
The level sets induced by the function $g$ are defined by
\[K_g(\alpha) = \left\{ x \in Y : g(x)= \alpha \right\}. \]
Since they are pairwise disjoint they induce the \emph{multifractal decomposition}
\[X = \left( X \setminus Y \right) \cup \bigcup_{\alpha \in [-\infty, +\infty]} K_g(\alpha) . \]
Let $G$ be a real function defined on the set of subsets of $X$. The \emph{multifractal spectrum} $\mathcal{S}:[-\infty, + \infty] \to \mathbb{R}$ is the function that encodes the decomposition given by $g$ by means of the function $G$, that is
\[\mathcal{S}(\alpha) = G(K_g(\alpha) ). \]
We stress that in this definition no dynamical system is involved. The functions $g$ that we will consider are related to the dynamics of a certain systems and are, in general, only measurable functions. Hence, the multifractal decomposition is rather complicated.
Given a multimodal map $f:I \to I$ (our dynamical system) the functions $g$ that we will consider in this paper are:

\begin{enumerate}
\item The \emph{Lyapunov exponent}, that is the function defined by
\[\lambda(x)= \lim_{n \to \infty} \frac{1}{n} \log |Df^n(x)|,\]
whenever the limit exits.

\item The \emph{pointwise dimension} of an equilibrium state $\mu$.
\end{enumerate}

The function $G$ we will consider here is the Hausdorff dimension.
Note that we could also use entropy as a way of measuring the size of sets.

\section{Preliminaries: thermodynamic formalism and inducing schemes}\label{sec:prelim erg}

In this section we will introduce some ideas from thermodynamic formalism.  Then we will discuss inducing schemes, and finally we bring these together in thermodynamic formalism for countable Markov shifts.

\subsection{Thermodynamic formalism} \label{sec:therm form}
Let $f$ be a  map of a metric space, denote by $\mathcal{M}_f$ the set of $f-$invariant probability measures. Let $\phi: I \to [-\infty, \infty]$ be a  \emph{potential}. The \emph{topological pressure} of $\phi$ with respect to $f$ is defined by
\begin{equation*}
P_f(\phi)=P(\phi) = \sup \left\{ h(\mu) + \int \phi \ d\mu :  \mu \in \mathcal{M}_f \textrm{ and } - \int \phi \ d\mu < \infty\right\},
\end{equation*}
where $ h(\mu)$ denotes the measure theoretic entropy of $f$ with respect to $\mu$.
The pressure function $t \to P(t \phi)$ is convex (see \cite[Chapter 9]{Walbook} and \cite{Kellbook} for this and other properties of the pressure).

A measure $\mu_{\phi} \in  \mathcal{M}_f$ is called an \emph{equilibrium state} for $\phi$ if it satisfies:
\[ P(\phi) = h(\mu_{\phi}) + \int \phi \ d\mu_{\phi}. \]

The following results regarding existence and uniqueness of equilibrium states and the regularity of the pressure function were proved in \cite{IT}.

\begin{teo} \label{thm:ITeq_exist_unique}
Let $f\in \F$.  Then there exists $t^+ \in (0, + \infty]$ such that if $t \in (-\infty , t^+)$ there exists a unique equilibrium measure $\mu_t$ for the potential $-t \log |Df|$. Moreover, the measure $\mu_t$ has positive Lyapunov exponent.
\end{teo}

We define the pressure function
$$p(t):=P(-t\log|Df|).$$

\begin{teo}
Let $f\in \F$.  Then  for $t^+$ as in Theorem~\ref{thm:ITeq_exist_unique}, if $t \in (-\infty , t^+)$ then the pressure function $t\mapsto p(t)$ is strictly convex and $C^1$.
\label{thm:IT press conv}
\end{teo}

For future use, for $\mu\in \M_f$ we define the Lyapunov exponent of $\mu$ as
$$\lambda(\mu)=\lambda_f(\mu):=\int\log|Df|~d\mu.$$

\subsection{Inducing schemes} \label{subsec:ind_sch}
A strategy used to study multimodal maps $f \in \mathcal{F}$, considering that they lack Markov structure and expansiveness, is to consider a generalisation of the first return map. These maps are expanding and are Markov (although over a countable alphabet). The price one has to pay is to loss compactness. The idea is to study the inducing scheme and then to translate the results into the original system.

We say that $(X,F,\tau)$ is an \emph{inducing scheme} for $(I,f)$ if
\begin{list}{$\bullet$}{\itemsep 0.2mm \topsep 0.2mm \itemindent -0mm \leftmargin=5mm}
\item $X$ is an interval containing a finite or countable
collection of disjoint intervals $X_i$ \st $F$ maps each $X_i$
diffeomorphically onto $X$, with bounded distortion (i.e. there
exists $K>0$ so that for all $i$ and $x,y\in X_i$, $1/K\le DF(x)/DF(y) \le K$);
\item $\tau|_{X_i} = \tau_i$ for some $\tau_i \in \N$ and $F|_{X_i} = f^{\tau_i}$.  If $x \notin \cup_iX_i$ then $\tau(x)=\infty$.
\end{list}
The function $\tau:\cup_i X_i \to \N$ is called the {\em inducing time}. It may
happen that $\tau(x)$ is the first return time of $x$ to $X$, but
that is certainly not the general case.  For ease of notation, we will frequently write $(X,F)=(X,F,\tau)$.  We denote the set of points $x\in I$ for which there exists $k\in \N$ such that $\tau(F^n(f^k(x)))<\infty$ for all $n\in \N$ by $(X,F)^\infty$.

Given $(I, f)$ and a potential $\phi$, the next definition gives us the relevant potentials for an inducing scheme for $f$.

\begin{defi}
Let $(X, F, \tau)$ be an inducing scheme for the map $f$.  Then for a potential $\phi:I\to \R$,  the induced potential $\Phi$ for $(X,F, \tau)$ is given by $$\Phi(x)=\Phi^F(x):=\phi(x)+\cdots +\phi\circ f^{\tau(x)-1}(x).$$
\end{defi}

Note that in particular for the  geometric potential $\log|Df|$, the induced potential for a scheme $(X,F)$ is $\log|DF|$.

Given an inducing scheme $(X,F, \tau)$, we say that a measure $\mu_F$ is a \emph{lift} of $\mu$ if for all $\mu$-measurable subsets $A\subset I$,
\begin{equation} \mu(A) = \frac1{\int_X \tau \ d\mu_F} \sum_i \sum_{k = 0}^{\tau_i-1} \mu_F( X_i \cap f^{-k}(A)). \label{eq:lift}
\end{equation}
Conversely, given a measure $\mu_F$ for $(X,F)$, we say that
$\mu_F$ \emph{projects} to $\mu$ if \eqref{eq:lift} holds.
We call a measure
$\mu$  \emph{compatible to} the inducing scheme $(X,F,\tau)$ if

\begin{list}{$\bullet$}{\itemsep 1.0mm \topsep 0.0mm \leftmargin=5mm}
\item $\mu(X)> 0$ and $\mu\left(X \setminus (X,F)^\infty\right) = 0$; and
\item there exists a measure $\mu_F$ which projects to $\mu$ by
\eqref{eq:lift}: in particular $\int_X \tau \ d\mu_F <
\infty$.
\end{list}

If the $f-$invariant measure $\mu$ has positive Lyapunov exponent then it can be lifted to an $F-$invariant measure $\mu_F$.

The following result, proved in \cite{T} (see also \cite{BTeqnat}) which we use frequently in the rest of the paper, provides our family of inducing schemes. We require the definition $\overline\lambda(x):= \limsup_{n \to \infty} \frac{1}{n} \log |Df^n(x)|$.

\begin{teo}\label{thm:schemes}
Let $f\in \F$.  There exist a countable collection $\{(X^n,F_n)\}_n$ of inducing schemes with $\bd X^n \notin (X^n,F_n)^\infty$ such that:
\newcounter{Mcount}
\begin{list}{\alph{Mcount})}{\usecounter{Mcount} \itemsep 1.0mm \topsep 0.0mm \leftmargin=5mm}
\item any ergodic invariant probability measure $\mu$ with $\lambda(\mu)>0$ is compatible with one of the inducing schemes $(X^n, F_n)$.  In particular there exists and ergodic $F_n$-invariant probability measure $\mu_{F_n}$ which projects to $\mu$ as in \eqref{eq:lift};
\item any equilibrium state for $-t\log|Df|$ where $t\in \R$ with $\lambda(\mu)>0$ is compatible with all inducing schemes $(X^n, F_n)$.
\item if $f\in \F_{ac}$ then for all $\eps>0$ there exists $N\in \N$ such that $$dim_H\Big(\left\{x\in I:\overline\lambda(x)>0\right\}\sm\left(\cup_{n=1}^N(X^n, F_n)^\infty\right) \Big)<\eps.$$
\end{list}
\end{teo}

\begin{rem}
In principle $\dim_H\left\{x\in I:\overline\lambda(x)\le0\right\}$ could be positive.  However, as in \cite[Corollary B]{T}, if $f$ is a map with exponential growth along critical orbits then $\overline\lambda(x)>0$ for all $x\in I$.
\end{rem}

If $(X, F, \tau)$ is an inducing scheme for the map $f$ with $\bd X\notin (X,F)^\infty$, then the system $F:(X,F)^\infty \to (X,F)^\infty$ is topologically conjugated to the full-shift on a countable alphabet. Hence we can transfer our study to those shifts.  We explain this in the next subsection.

\subsection{Countable Markov shifts} \label{countable}
Let $\sigma \colon \Sigma \to \Sigma$ be a one-sided Markov shift
with a countable alphabet $S$. We equip $\Sigma$ with the
topology generated by the cylinder sets
\[ C_{i_0 \cdots i_n}= \{x \in
\Sigma : x_j=i_j \text{ for } 0 \le j \le n \}.\] Given a function
$\phi\colon \Sigma \to\R$, for each $n \geq 1$ we set
\[
V_{n}(\phi) = \sup \left\{|\phi(x)-\phi(y)| : x,y \in \Sigma,\
x_{i}=y_{i} \text{ for } 0 \le i \le n-1 \right\}.
\]
We say that $\phi$ has \emph{summable variation} if
$\sum_{n=2}^{\infty} V_n(\phi)<\infty$. Clearly, if $\phi$ has
summable variation then it is continuous.
The so-called \emph{Gurevich pressure} of $\phi$ was defined by Sarig \cite{Sartherm}
as
\[
 P_G(\phi) := \lim_{n \to
\infty} \frac{1}{n} \log \sum_{x:\sigma^{n}x=x} \exp \left(
\sum_{i=0}^{n-1} \phi(\sigma^{i}x)\right) \chi_{C_{i_{0}}}(x),
\]
where $\chi_{C_{i_{0}}}(x)$ is the characteristic function of the
cylinder $C_{i_{0}} \subset \Sigma$.
We consider a special class of invariant measures. We say that $\mu\in \mathcal{M}_\sigma$ is a \emph{Gibbs measure} for the
function $\phi \colon \Sigma \to \R$ if for some constants $P$,
$C>0$ and every $n\in \N$ and $x\in C_{i_0 \cdots i_n}$ we have
\begin{equation} \label{eq:gibbs}
\frac{1}C \le \frac{\mu(C_{i_0\cdots i_n})}{\exp (-nP + \sum_{i=0}^n
\phi(\sigma^k x))} \le C.
\end{equation}
It was proved by Mauldin and Urba\'nski \cite{muGIBBS} and by Sarig in \cite{SarBIP} that if $(\Sigma, \sigma)$ is a full-shift and the
function $\phi$ is of summable variations with finite Gurevich pressure, then it has an invariant Gibbs measure. Moreover, if $(\Sigma, \sigma)$ is the full-shift and
$\phi$ has summable variations and finite Gurevich pressure then the function
\[t \mapsto P_G(t \phi)\]
is real analytic for every $t \geq 1$ (see \cite{Sarphase}).

\begin{rem}
Recall that the system $F:(X,F)^\infty \to (X,F)^\infty$ is topologically conjugated to the full-shift on a countable alphabet. In particular, every potential $\Phi: X \to \R$ has a symbolic version, $\overline{\Phi}: \Sigma \to \R$.  In all the cases of induced systems we consider in this paper we have, by the Variational Principle \cite[Theorem 3]{Sartherm}, $P(\Phi)= P_G(\overline{\Phi})$.  Therefore in order to simplify the notation we will denote the pressure by $P(\Phi)$ when the underlying system is the induced system and when it is the full-shift on a countable alphabet.
\end{rem}

\section{The Lyapunov spectrum}

\label{sec:Lyap spec}

In this section we consider the multifractal decomposition of the interval obtained by studying the level sets associated to the Lyapunov exponent for maps $f\in \F$.   In recent years a great deal of attention has been paid to this decomposition.  This is partly due to the fact that the Lyapunov exponent is dynamical characteristic that captures important features of the dynamics.  It is closely related to the existence of absolutely continuous (with respect to Lebesgue) invariant measures.

The \emph{lower/upper pointwise Lyapunov exponent} are defined by  $$\underline\lambda_f(x):=\liminf_{n\to\infty} \frac{1}{n} \sum_{j=0}^{n-1} \log|Df(f^j(x))|, \text{ and } \overline\lambda_f(x):=\limsup_{n\to\infty} \frac{1}{n} \sum_{j=0}^{n-1} \log|Df(f^j(x))|$$ respectively.  If $\underline\lambda_f(x)=\overline\lambda_f(x)$, then the \emph{Lyapunov exponent} of the map $f$ at $x$ is defined by
$\lambda(x)=\lambda_f(x)= \underline\lambda_f(x)=\overline\lambda_f(x)$.

The associated level sets for $\alpha \geq 0$ are defined by,
\begin{equation*}
J(\alpha)= \Big{\{} x \in I :  \lim_{n \to \infty} \frac{1}{n} \log |(f^n)'(x)| = \alpha \Big{\}}.
\end{equation*}
 Note that for some values of $\alpha$ we have  $J(\alpha)=\emptyset$, a trivial example being for $\alpha>\log(\sup_{x\in I}|Df(x)|)$. Let
\begin{equation*}
J'= \Big{\{} x \in I : \textrm{the limit} \lim_{n \to \infty} \frac{1}{n} \log |(f^n)'(x)| \textrm{ does not exist}  \Big{\}}.
\end{equation*}
The unit interval can be decomposed in the following way (the
\emph{multifractal decomposition}),
\begin{equation*}
 [0,1]= J' \cup \left( \cup_{\alpha} J(\alpha) \right).
\end{equation*}
The function that encodes this decomposition is called \emph{multifractal spectrum of the Lyapunov exponents} and is defined by
\begin{equation*}
L(\alpha):= \dim_H(J(\alpha)),
\end{equation*}

This function was first studied by Weiss \cite{Wei}, in the context of  axiom A maps.  The study of the multifractal spectrum of the Lyapunov exponent for multimodal maps began with the work of Todd \cite{T}.

We define
\begin{eqnarray*}
\lambda_m:= \inf \left\{\lambda(\mu) : \mu \in \mathcal{M} \right\} \text{ and } \lambda_M:= \sup \left\{\lambda(\mu) : \mu \in \mathcal{M} \right\}.
\end{eqnarray*}
We next show that the range of values that the Lyapunov exponent can attain is an interval contained in $[\lambda_m, \lambda_M]$.

We define
$$\lambda_{\inf}:=\inf\{\lambda(x):x\in I \text{ and this value is defined}\}$$
and
$$\lambda_{\sup}:=\sup\{\lambda(x):x\in I \text{ and this value is defined}\}.$$
The next lemma follows as in \cite[Lemma 6]{GelPrRa}.

\begin{lema}\label{lem:per pt small LE}
For all $x\in (\lambda_{\inf}, \lambda_{\sup})$ there exists a periodic point $p$ with $\lambda(p)\le \lambda(x)$.
\end{lema}

This leads to:

\begin{lema}\label{lem:LE range}
$\lambda_m=\lambda_{\inf}$ and $\lambda_M=\lambda_{\sup}$.
\end{lema}

\begin{proof}
First note that $\lambda_m\ge \lambda_{\inf}$ is clear from the definition.  To show the opposite inequality, suppose that $\lambda(x)$ is defined, in which case by definition $\lambda(x)\ge\lambda_{\inf}$.  Then by Lemma~\ref{lem:per pt small LE}, there exists a periodic point $p$ with $\lambda(p)\le \lambda(x)$.  The Dirac measure $\mu_p$ equidistributed on the orbit of $p$ must have $\lambda(\mu_p)\in [\lambda_m, \lambda_M]$, so $\lambda(x)\ge \lambda_m$.  Hence $\lambda_m\le \lambda_{\inf}$, as required.

For the second part, again suppose that $\lambda(x)$ is well defined.  Then let $$\mu_n:=\frac1n\sum_{k=0}^{n-1}\delta_{f^k(x)}$$ where $\delta_y$ is the Dirac mass on $y$.  Let $\mu$ be a weak$^*$ limit of this sequence.  Since $\log|Df|$ is upper semicontinuous, $$\lambda_M\ge \lambda(\mu)\ge \lim_{n\to \infty}\lambda(\mu_n)=\lambda(x)$$ as required.
\end{proof}

 Note that one way of reading the previous lemma is that if we understand the sets  $J(\alpha)$ for $\alpha\in [\lambda_m, \lambda_M]$ then we understand all non-empty sets $J(\alpha)$.  That is, the spectrum is complete in the sense of Schmeling (see \cite{sch})

It can be shown by Theorem~\ref{thm:ITeq_exist_unique} that for every $\lambda \in (\lambda_m, \lambda_M]$ there exists a unique parameter $t_{\lambda} \in \mathbb{R}$ with a unique equilibrium measure $\mu_{\lambda}$ corresponding to $-t_{\lambda} \log |Df|$ such that
\[ \lambda(\mu_{\lambda}) = \lambda. \]
However, as in \cite[Lemma 5.5]{BrKell} there are maps $f\in\F$ with no measure $\mu\in \M_f$ with $\lambda(\mu)=\lambda_m=0$.

We define
\[
A:=
\begin{cases}
\left[\lambda_m,-D^-p(t^+)\right) & \text{ if } t^+\neq 1,\\
\{0\} & \text{ if } t^+=1.
\end{cases}
\]

\begin{rem}
If we assumed that $f\in \F_{ac}$ was $C^2$, then combining the arguments of \cite[Exercise V.1.4]{MSbook} and \cite[Proposition 7]{SVarg}, the acip $\mu_{ac}$ must have positive entropy and Lyapunov exponent.   Since in this case $Dp^-(1)=-\lambda(\mu_{ac})$, $p(t)$ is strictly decreasing on $(-\infty, 1)$.  In this case, an equivalent definition of $A$ would be
\[
A:=
\begin{cases}
\left[\lambda_m,-D^-p(t^+)\right) & \text{ if } \lambda_m>0,\\
\{0\} & \text{ if } \lambda_m=0.
\end{cases}
\]
\label{rmk:acip pos ent}
\end{rem}

\begin{teo}  Suppose that $f\in \F_{ac}$.  Let $\lambda \in \R\sm A$. The Lyapunov spectrum satisfies the following relation
\begin{equation} \label{eq:LE leg}
L(\lambda) = \frac{1}{\lambda} \inf_{t \in \mathbb{R}}\left(p(t) + t \lambda\right) =
 \frac{1}{\lambda} \left(p(t) +t_{\lambda} \lambda \right).
 \end{equation}
If $\lambda \in (-D^-p(t^+), \lambda_M)$ then we also have
\begin{equation}
L(\lambda)= \frac{h(\mu_{\lambda})}{\lambda}.
\end{equation}
If $t^+>1$ and $\lambda\in A$ then $$L(\lambda)\ge \frac{1}{\lambda} \inf_{t \in \mathbb{R}}\left(p(t) + t \lambda\right).$$
Moreover, the irregular set $J'$ has full Hausdorff dimension.
\label{thm:Lyap detail}
\end{teo}

Theorem~\ref{thm:main Lyap} follows immediately from this.

\begin{rem}
 Theorem~\ref{thm:Lyap detail} implies that if $f\in \F_{ac}$ is such that $t^+=1$ then for every $\lambda \in (0,-D^-p(1))$ we have that $L(\lambda)=1$.  Recall that in the unimodal case, as in \cite{NoSa}, $\lambda_m=0$ implies that the Collet-Eckmann condition fails.  In the case that $t^+\neq 1$ (for example when $\lambda_m>0$) we expect that the formula $L(\lambda)= \frac{1}{\lambda} \inf_{t \in \mathbb{R}}\left(p(t) + t \lambda\right)$ still holds for $\lambda \in [\lambda_m,-D^-p(t^+))$, but we do not find an upper bound on this value in this paper.
\label{rem:lack of LE upper bd}
\end{rem}

\begin{rem}
We stress that the above formula \eqref{eq:LE leg} for $L(\lambda)$ does not imply that the Lyapunov spectrum is concave.  For a discussion on that issue see the work of Iommi and Kiwi \cite{IomKi}.
\end{rem}

\begin{proof}[Proof of the lower bound]
Let $\lambda \in (-D^-p(t^+), \lambda_M)$. In order to prove the lower bound on the formula \eqref{eq:LE leg}, consider the equilibrium measure $\mu_{\lambda}$ corresponding to $-t_{\lambda} \log |Df|$ such that $ \lambda(\mu_{\lambda}) = \lambda. $ We have
\begin{enumerate}
\item $\mu_{\lambda}(I\sm J(\lambda))= 0$;
\item the measure $\mu_{\lambda}$ is ergodic;
  \item  by \cite{Hofdim}, the pointwise dimension is $\mu_{\lambda}$-almost everywhere equal to
\[\lim_{r \to 0} \frac{\log \mu_{\lambda}(B(x,r))}{\log r}   = \frac{h(\mu_{\lambda})}{\lambda},\]
where $B(x,r)$ is the ball of radius $r>0$ centred at $x \in [0,1]$.
\end{enumerate}
Therefore, Proposition~\ref{prop:lower} implies
\[ \dim_H(J(\lambda)) \geq  \frac{h(\mu_{\lambda})}{\lambda}.\]
 We next consider $\lambda \in (\lambda_m , -D^-p(t^+) )$.  The following lemma applies when $\lambda_m=0$.

\begin{lema} \label{lem:small LE big dim}
Suppose that $f\in \F_{ac}$ has $t^+=1$.  Then for any  $\alpha\in (0, -D^-p(1))$ and $\eps>0$ there exists an ergodic measure $\mu\in \M$ with $\lambda(\mu)=\alpha$ and $\dim_H(\mu)\ge 1-\eps$.
\end{lema}

 The proof follows by approximating $(I,f)$ by hyperbolic sets on which we have equilibrium states with small Lyapunov exponent and large Hausdorff dimension.  The hyperbolic sets are invariant sets for truncated inducing schemes.

\begin{proof}    We may assume that $D^-p(t^+)<0$, otherwise there is nothing to prove.
Let $\eps'>0$.  By Lemma~\ref{lem:per pt small LE}, we can choose $\alpha'\in (0, \alpha)$, depending on $\eps'>0$, with a periodic point $x_p\in I$ of period $p$ such that $\lambda(x)=\alpha'$.  Then there is an inducing scheme $(X,F, \tau)$ as in Theorem~\ref{thm:schemes} such that $\tau(x_p)= p$ and $\tau(x)\ge p$ for all $x\in X$.   This follows from the fact that if $X$ is a small neighbourhood of $x_p$ then all points in $X$ must shadow the orbit of $x_p$ for at least the first $p$ iterates.  We can truncate $(X, F, \tau)$ to a scheme with $N$ branches $(\tilde X^N, \tilde F_N, \tilde \tau_N)$ and define
$$p_N(t):=\sup\{h(\mu)-t\lambda(\mu)-p(t):\mu\in\M \text{ and } \mu \text{ is compatible with } (\tilde X^N,\tilde F_N)\}.$$

\begin{claim}
There exist $\delta(N)>0$ where $\delta(N) \to 0$ as $N\to \infty$ such that for $t\in(1-\eps', 1]$, $p_N(t)\ge -\delta(N)$ and for $t\in (1, 1+\eps')$, $p_N(t)\ge -t(\alpha'+\delta(N))$.
\end{claim}

\begin{proof}
By Theorem~\ref{thm:schemes}, $P(-t\log|Df|-\tau p(t))=0$ and indeed $p_N(t)\to p(t)$ for all $t\le 1$, so the first part of the claim follows.

We now suppose that $t\ge1$.  As in \cite[Proposition 2.2]{IT}, $P(-t\log|DF|)\le 0$.  Moreover, since the equidistribution on the orbit of $x_p$ lifts to $(X,F)$, the Variational Principle implies that $P(-t\log|DF|) \ge -tp\alpha'$.
For $\delta>0$, again as in \cite[Proposition 2.2]{IT} we can find $N$, so that the scheme $(\tilde X^N, \tilde F_N, \tilde \tau_N)$ has $P(-t\log|D\tilde F_N|)\ge -tp(\alpha'+\delta)$.  Therefore an equilibrium state $\mu_{F, t,N}$ for
$(\tilde X^N, \tilde F_N, \tilde \tau_N)$ and $-t\log|D\tilde F_N|$ will have
$$h(\mu_{F, t,N})-t\int\log|D\tilde F_N|~d\mu_{F, t,N} \ge -tp(\alpha'+\delta).$$
Furthermore, by the setup, $\int\tau~d\mu_{F, t,N} \ge p$.  So projecting to $\mu_{t, N}$, a measure on the original scheme, we have
$$h(\mu_{t,N})-t\int\log|Df|~d\mu_{t,N} \ge \frac{-tp(\alpha'+\delta)}{
\int \tau~d\mu_{F, t,N}} \ge -t(\alpha'+\delta),$$
proving $p_N(t) \ge -t(\alpha'+\delta)$.
\end{proof}

The claim implies that $Dp_N(t) \to Dp(t)$ for $t<1$ also.  Since for $t<1$  we have $-Dp(t)=\lambda(\mu)>\alpha$ for $\mu$ the equilibrium state for $-t\log|Df|$, the claim implies that for $\eps'>0$ as above there exists $N$ such that there is $t\in [1, 1+\eps')$ with $-Dp_N(t)=\alpha$ and also $p_N(t) >-\eps'$.  This implies that there is an equilibrium state $\mu_\alpha$ which is a projection of $(X, F, -t\log|D\tilde F_N|)$ such that
$$h(\mu_\alpha)-\alpha\ge -\eps'.$$
Hence $L(\alpha) \ge \dim_H(\mu_\alpha) =h(\mu_\alpha)/\alpha \ge 1-\eps'/\alpha$.  The proof of the lemma concludes by setting $\eps':=\eps\alpha$.
\end{proof}

For the case where $t^+\in (1,\infty)$ and $p$ is not $C^1$ at $t=t^+$ we can apply a similar argument.  We showed in \cite[Remark 9.2]{IT} that $t^+\in (1,\infty)$ implies that $\lambda(\mu) \neq \lambda_m$ for all $\mu\in \M$.  Therefore we can use the fact that for any $\eps'>0$ there exists $\mu\in \M$ such that $\lambda(\mu)\in (\lambda_m, \lambda_m+\eps')$.  In this case we obtain the lower bound:
$$L(\alpha) \ge t^++\frac{p(t^+)}\alpha,$$
as required.
 \end{proof}

\begin{proof}[Proof of the upper bound]

In the case $\lambda_m=0$ and $\alpha\in (0, -D^-p(1))$ we showed $L(\alpha)\ge 1$, so in fact $L(\alpha)=1$.  So to complete the proof of Theorem~\ref{thm:Lyap detail}, we will prove the upper bound for $L(\alpha)$ when $\alpha\in [-D^-p(1), \lambda_M]$ and $\lambda_m$ is any value.

Let $(X, F, \tau)$ be an inducing scheme for the map $f$.  Note that the $(X,F)$ is topologically conjugated to the full-shift on a countable alphabet.
Recall that (see Section \ref{countable}) every potential $\phi: X \to \mathbb{R}$ of summable variations  and finite pressure has a Gibbs measure \cite{SarBIP}.


\begin{rem}
Note that if $\mu_t$ is the equilibrium measure for $-t \log |Df|$ then the lifted measure $\mu_{F,t}$ is the Gibbs measure corresponding to the potential $\Phi_{t}= -t \log|DF| -P(-t \log |Df|) \tau$.  Note that $\Phi_t$ has summable variations by, for example, \cite[Lemma 8]{BTeqnat}.
\end{rem}

For an inducing scheme $(X^n,F_n,\tau_n)$ constructed as in the
proof of Theorem~\ref{thm:schemes}, consider the level set
\[J_n(\lambda) := \left\{ x \in X^n : \lim_{k \to \infty} \frac{\sum_{j=0}^{k-1} \log |DF_n(F_n^j(x))|}{\sum_{j=0}^{k-1} \tau_n(F_n^jx)} = \lambda \right\}. \]

\begin{rem}
If $\mu_t$ is the equilibrium measure for $-t \log |Df|$ and $\lambda(\mu_t) = \lambda$ then the lifted measure $\mu_{F_n,t}$ is such that
\[\mu_{F_n,t} \left( I \sm J_n(\lambda)   \right)  = 0 .\]
\end{rem}
Denote by $I_k^n(x)$ the cylinder (with respect to the Markov dynamical system $(X^n,F_n)$ of length $k$ that contains the point $x \in X$ and by
$|I_k^n(x)|$ its Euclidean length. By definition there exists a positive constant $K>0$ such that for every $x \in X$ which is not the  preimage of a boundary point and every $k \in \mathbb{N}$ we have
\[ \frac{1}{K} \leq \frac{|I_k^n(x)|}{|(F_n^k)'(x)|} \leq K.        \]

\begin{defi}
For an inducing scheme $(X^n,F_n)$ and a point $x\in X^n$ not a preimage of a boundary point of $X^n$, we define the  \emph{Markov pointwise dimension of $\mu_{F_n,t}$ at the point $x$} as
$$\delta_{\mu_{F_n,t }}(x) := \lim_{k \to \infty} \frac{\log \mu_{F_n,t}(I_k^n(x))}{-\log |I_k^n(x)|}$$
if this limit exists.
\end{defi}

\begin{lema}
The Hausdorff dimension of $J_n(\lambda)$ is given by
\[\dim_H(J_n(\lambda) ) = \frac{h(\mu_t)}{\lambda}=\delta_{\mu_{F_n,t }}(x)\]
for $\mu_{F_n,t }$-a.e. $x\in X^n$.
 \end{lema}

\begin{proof}
Let $x \in J_n(\lambda)$ and $\mu_{F_n,t}$ be the Gibbs measure with respect to $\Phi_{t, n}:=-t \log|DF_n| -P(-t \log |Df|) \tau_n$. Since we have bounded distortion, the Markov pointwise dimension of $\mu_{F_n,t}$ at the point $x\in X^n$, if it exists, is
\begin{align*}
\delta_{\mu_{F_n,t}}(x) &= \lim_{k \to \infty} \frac{\log \mu_{F_n,t}(I_k^n(x))}{-\log |I_k^n(x)|} = \lim_{k\to \infty} \frac{\sum_{i=0}^{k-1}\Phi_{t,n}(F_n^i(x))}{-\log|DF_n^k(x)|}\\
&= \lim_{k \to \infty} \frac{-t \log |DF_n^k(x)|-P(-t \log |Df|)\sum_{i=0}^{k-1} \tau_n(F_n^i (x))}{-\log |DF_n^k(x)|} \\
&=t + P(-t \log |Df|) \lim_{k \to \infty} \frac{\sum_{i=0}^{k-1}  \tau_n(F_n^i (x))}{\log |DF_n^k(x)|}\\
&= t+ \left( h(\mu_t) -t  \int  \log |Df|  d\mu_t \right)  \lim_{k \to \infty} \frac{\sum_{i=0}^{k-1}  \tau_n(F_n^i (x))}{\log |DF_n^k(x)|}. \end{align*}
But since $x \in J_n(\lambda)$ we have that
\[  \lim_{k \to \infty} \frac{\sum_{i=0}^{k-1}  \tau_n(F_n^i (x))}{\log |DF_n^k(x)|} = \lambda.\]
Therefore,
\[\delta_{\mu_{F_n,t} }(x) = t + \frac{h(\mu_t) -t \lambda}{\lambda}= \frac{h(\mu_t)}{\lambda}.\]
The following result was proved by Pollicott and Weiss \cite{PolWe}. Suppose that $\delta_{\mu_{F_n,t}}(x)$ and $\lambda(x)$ exist, then
\[d_{\mu_{F_n,t}}(x) = \delta_{\mu_{F_n,t}}(x).\]
Therefore, we have that for every point $x \in J_n(\lambda)$ the pointwise dimension is given by
\[d_{{\mu_{F_n,t}}}(x) =  \frac{h(\mu_t)}{\lambda}.\]
Since  $\mu_{F_n,t}(X^n \sm J_n(\lambda))=1$ we have that
\[\dim_H(J_n(\lambda) )= \frac{h(\mu_t)}{\lambda},\]
as required.
\end{proof}

Note that the projection map $\pi_n : X^n \to I$ from each inducing scheme $(X^n, F_n)$ into the interval $I$ is a  bilipschitz map. Therefore,
\[\dim_H(\pi_n (J_n(\lambda))) = \frac{h(\mu_t)}{\lambda}.\]
By Theorem~\ref{thm:schemes} c),
we obtain the desired upper bound,
\begin{align*}
\dim_H(J(\lambda))
& \leq \sup_n\left\{\dim_H\left(\pi (J_n(\lambda))\right)\right\} \leq
\frac{h(\mu_t)}{\lambda}.\end{align*}
\end{proof}

\begin{rem}
It is a direct consequence of the results of Barreira and Schmeling \cite{BarSc} that the set $J'$ has full Hausdorff dimension.
\end{rem}

\section{The pointwise dimension spectrum}
\label{sec:dim spec intro}

 In this section we give details of the notions of multifractal spectrum of the pointwise dimension of equilibrium states.   As in \cite{T}, this can be seen as a generalisation of the results on the Lyapunov spectrum. The \emph{pointwise dimension} of the measure $\mu$ at the point $x \in I$ is defined by
\begin{equation} \label{puntual}
 d_{\mu}(x) := \lim_{r \rightarrow 0} \frac{ \log \mu((x-r,x+r))}{\log r},
 \end{equation}
provided the limit exists.  This function describes the power law behaviour of the measure of an interval,
\[\mu((x-r, x+r)) \sim r^{d_{\mu}(x)}.\]
The pointwise dimension induces a decomposition of the space  into level sets:
\[ K(\alpha) = \{ x \in \Sigma : d_{\mu}(x) = \alpha \} \textrm{ , }K' =\{ x \in \Sigma : \textrm{the limit } d_{\mu}(x)  \textrm{ does not exist} \}. \]
The set $K'$ is called the \emph{irregular set}. The decomposition:
\[ I = \Big( \bigcup_{\alpha}K(\alpha) \Big) \bigcup K' \]
is called the \emph{multifractal decomposition}. The \emph{multifractal spectrum of pointwise dimension} is defined by
\[ \dimspec_{\mu}(\alpha)= \dim_{H}(K(\alpha)).\]

Let us stress that for maps $f \in \mathcal{F}$, points which are not `seen' by inducing scheme are beyond our analysis.  However, as in \cite{T}, this set of points can often be shown to be negligible. This is indeed the case, Rivera-Letelier and Shen \cite{rivshen} have recently proved that
\[\dim_H \left\{ x \in [0,1] : \overline{\lambda}_f(x) =0 \right\} =0, \]
so that by Theorem~\ref{thm:schemes}(c), our inducing schemes capture all sets of positive Hausdorff dimension.


In order to describe the function $\dimspec_{\mu}$ we will study an auxiliary function:  the so-called \emph{temperature function} is defined in terms of the thermodynamic formalism and shown to be the Legendre-Fenchel transform of the multifractal spectrum.

\section{The temperature function} \label{sec:temp}

In this Section we study the  \emph{temperature function} which allows us to describe the multifractal spectrum.  Firstly we need to establish the existence of the measures that we are going to analyse. The measures we will study will be equilibrium states.  The class of potentials we consider is
$$\P:=\big\{\phi:I \to [\phi_{\min}, \phi_{\max}] \text{ for some } \phi_{\min}, \phi_{\max}\in (-\infty, 0) \text{ and } P(\phi)=0\big\}.$$
Note that any bounded potential $\phi'$ with $\phi'<P(\phi')$ can be translated into this class by setting $\phi:=\phi'-P(\phi')$.  Any equilibrium state for $\phi'$ is an equilibrium state for $\phi$.

We let $\P_H\subset \P$ be the set of H\"older potentials on $I$.
It is well known  (see for example \cite[Section 4]{Kellbook}) that potentials in $\P$ have (potentially many) equilibrium states with positive entropy.  Theorem~\ref{thm:eq cts pots} shows this.

We will not give the details of the proof of Theorem~\ref{thm:eq cts pots} since it can be proved as in Theorem~\ref{thm:unique mu_q}.  However, note that it follows using the inducing techniques as in \cite[Section 5]{IT}.  Note that by \cite[Lemma 3]{BTeqgen}, the H\"older condition on $\phi$ guarantees the summable variations for the inducing schemes.

As in the introduction, the \emph{temperature function with respect to $\phi$} is the function $T_\phi:\mathbb{R} \to \mathbb{R} \cup \{ \infty \}$ implicitly defined, for $q \in \mathbb{R}$, by the equation
\[T_\phi(q)= \inf \{ t \in \mathbb{R} : P(-t \log |Df| +q \phi) =0\}. \]
If for a fixed $q$ and for every $t \in \mathbb{R}$ we have that $P(-t \log |Df| + q \phi) >0$ then $T_\phi(q)= \infty$.
If there exists a finite number $$q_\infty:=\sup\{q\in \R:T_\phi(q)=\infty\},$$ then we say that $T_\phi$ has an \emph{infinite phase transition} at $q_\infty$.

\begin{rem}
Note that for $\phi\in \P$ we have $T_\phi(1)=0$, since by definition $P(\phi)=0$.  Moreover, $T_\phi(0)$ is the smallest root of the Bowen equation $P(-t \log |Df|)=0$.  It follows from the statement of \cite[Theorem 10.8(3)]{AvLyu} that there are unimodal maps in $\F$ with critical order $\ell_c>2$ for which $T_\phi(0)<1$.  This phenomenon is associated to the presence of a `wild attractor'.  For any unimodal map with quadratic critical point (i.e. $\ell_c=2$), there is no wild attractor and we have $T_\phi(0)=1$.  This is also true for any map in $\F_{ac}$.
\end{rem}

\begin{rem}
Note that $q'\le q$ implies $T_\phi(q')\ge T_\phi(q)$.  Therefore if $T_\phi$ has a phase transition at $q_\infty$ then $T(q)=\infty$ for all $q<q_\infty$.
\end{rem}

\begin{eje}[Regular]
Let $f:I \to I$ be a Collet-Eckmann unimodal map. Then the pressure function $t \to p(t)$ is strictly decreasing as in Theorem~\ref{thm:IT press conv}.  Moreover, $p$ is $C^1$ in an interval $(-\infty, t^+)\supset [0,1]$.  Consider the potential $\phi=-\htop(f)$ (that is, minus the topological entropy of the map $f$). In this case the function $T_\phi(q)$ is obtained by the equation in the variable $t \in \mathbb{R}$ given by \[P(-t \log |Df|)= q \htop(f).\]
For every $q \in \mathbb{R}$ this equation has a unique solution. Moreover, for $q$ in a neighbourhood of $[0,1]$, by Theorem~\ref{thm:ITeq_exist_unique} there exists a unique equilibrium state $\mu_{\phi_q}$ corresponding to the potential $\phi_q=-T_\phi(q) \log |Df| -q \htop(f)$.
\end{eje}

\begin{eje}[Infinite phase transition] \label{infinite_phase}
If $f\in \F$ is a unimodal map which is not Collet-Eckmann, then results in \cite{NoSa} imply
\[
P(-t\log |Df|)=
\begin{cases}
\text{positive} & \text{ if } t < 1,\\
0 & \text{ if } t \ge 1.
\end{cases}
\]
If we consider the constant potential $\phi:= -\htop(f)$ then for every $q <0$ we have that
\[P(-t \log |Df| +|q| \htop(f) ) =P(-t \log |Df|) +|q| \htop(f) \geq |q| \htop(f) > 0.\]
That is,
\[
T_\phi(q)=
\begin{cases}
\text{infinite} & \text{ if } q<0,\\
finite & \text{ if } q \ge 0.
\end{cases}
\]
In this case the function $T_\phi(q)$ has an infinite phase transition.
\end{eje}
We stress that infinite phase transitions can only occur at $q=0$. This is contained in the following proposition where we collect some basic properties of $T_\phi$.

\begin{prop}\label{prop:phase trans at 0}
Suppose that $f\in \F$ and $\phi\in \P$.  Then
\begin{enumerate}[a)]
\item $T_\phi(q)\in \R$ for all $q\ge 0$;
\item the function $T_\phi(q)$ can only have an infinite phase transition at $q_\infty=0$;
\item $T_\phi$, when finite, is strictly decreasing.
\end{enumerate}
\end{prop}

 We will use the following two Lemmas.
\begin{lema} \label{lem:T finite q pos}
Suppose that $f\in \F$ and $\phi\in \P$.  If $q \geq 0$ then the function $T_\phi(q)$ is finite.
\end{lema}

\begin{proof}
If $q \geq 0$ then $q\phi \leq 0$. Therefore,
\[P(-t \log |Df| +q \phi) \leq P(-t \log |Df|).\]
Since $P(- \log |Df| ) \leq 0$ and $t\mapsto P(-t\log|Df|)$ is decreasing, this implies that $T_\phi(q)\le 1$.  It remains to check  $T_\phi(q)\neq -\infty$.

We have
\[P(-t \log |Df| +q \phi) \geq P(-t \log |Df| +q\phi_{\min}) = P(-t \log |Df|) +q \phi_{\min}.\]
Since
\[\lim_{t \to -\infty} P(-t \log |Df|) = \infty, \]
there exists $t_0 <0$ such that
\[P(-t_0 \log |Df|)- q \phi_{\min} > 0.\]
That is
\[ P(-t_0 \log |Df| + q \phi) >0.\]
Since the function $t \to P(-t \log |Df| + q \phi)$ is continuous, the intermediate value theorem implies that there exists $T_\phi(q) \in (t_0, 1]$ such that
\[T_\phi(q) = \inf \{ t \in \mathbb{R} : P(-t \log |Df| +q \phi) = 0 \}, \]
as required.
\end{proof}

\begin{lema} \label{lem:T finite}
Suppose that $f\in \F$ and $\phi\in \P$.
If $$\lim_{t \to +\infty} P(-t \log |Df|)= - \infty$$ then $T_\phi(q)$ is finite for every $q \in \mathbb{R}$.
\end{lema}

\begin{proof}
We will show that $T_\phi(q)$ must lie in a finite interval.
First note that if $q <0$ then
\[P(-t \log |Df| + q \phi) \leq P(-t \log |Df|) + q \phi_{\min}.\]
So by assumption if we take $t_1>0$ large enough we have that
\[P(-t_1 \log |Df| + q \phi) \leq 0.\]  From the other side, as in the proof of Lemma~\ref{lem:T finite q pos} we can find $t_0\in \R$ such that
$$P(-t_0 \log |Df| + q \phi) > 0.$$
Hence $T_\phi(q)$ lies in the finite interval $(t_0, t_1]$.

The case of positive $q$ is handled by
Lemma~\ref{lem:T finite q pos}.
\end{proof}

\begin{proof}[Proof of Proposition~\ref{prop:phase trans at 0}]

\emph{Part a):} This follows immediately from Lemma~\ref{lem:T finite q pos}.

\emph{Part b):}
Lemma~\ref{lem:T finite} implies that if $\lim_{t \to +\infty} P(-t \log |Df|)= - \infty$ then we can not have an infinite phase transition.  Therefore, adding this to Lemma~\ref{lem:T finite q pos}, to prove part 1 of the proposition we only need to examine the case when
the limit is finite: $\lim_{t \to +\infty} P(-t \log |Df|) > - \infty$ and $q<0$.

By definition, $Dp(t)\le -\lambda_mt$ so the only way that we can have the following bound $\lim_{t \to +\infty} P(-t \log |Df|) > - \infty$ is if $\lambda_m= 0$ (note that $\lambda(\mu)\ge 0$ for all $\mu\in \M$ by \cite{Prz}).  This implies $P(-t \log |Df|) \ge 0$ for all $t \in \R$.  Now suppose that $q<0$.  Then
$$P(-T(q)\log|Df|+q\phi)\ge P(-T(q)\log|Df|) + q\phi_{\max}\ge q\phi_{\max}>0.$$  Hence $T(q)=\infty$.  Since this holds for all negative $q$, the infinite phase transition must occur at 0.

\emph{Part c):}
Let $q\in \R$ and $\delta>0$.  Then
\begin{align*}
P(-T_\phi(q)\log|Df|+(q+\delta)\phi) &\le P(-T_\phi(q)\log|Df|+q\phi)+\delta\phi_{\max}\\
&< P(-T_\phi(q)\log|Df|+q\phi).
\end{align*}
So there is no way that $T_\phi(q)$ can be $T_\phi(q+\delta)$, proving part 3.
\end{proof}

In the next theorem  we establish the existence of equilibrium measures for the potential
$$\phi_q:= -T(q) \log |Df| +q \phi$$ for a maximal range of values of the parameter $q \in \R$.
The strategy of the proof follows the arguments developed in \cite{IT} to prove the existence and uniqueness of equilibrium measures for the geometric  potential $-t \log |Df|$.


We define the constants $q_\phi^-\le q_\phi^+$ as follows:
\begin{list}{$\bullet$}{\itemsep 0.2mm \topsep 0.2mm \itemindent -0mm \leftmargin=5mm}
\item $q_\phi^+$ is defined, if possible, to be the infimum of  $q\ge 1$ such that there exists $\eps_q>0$ such that for all $\eps\in(0, \eps_q)$ there exists $\delta>0$ such that for any $\mu\in \M$,
$$\left|h(\mu)+\int\phi_q~d\mu\right|<\delta \quad \text{ implies }\quad  h(\mu)>\eps.$$
If there is no such value, then $q_\phi^+:=\infty$.
\item
If $T_\phi$ has an infinite phase transition then $q_\phi^-:=0$.  If not then, if possible, it is defined as being the supremum of $q\le 1$ such that there exists $\eps_q>0$ such that for all $\eps\in(0, \eps_q)$ there exists $\delta>0$ such that for any $\mu\in \M$,
$$\left|h(\mu)+\int\phi_q~d\mu\right|<\delta \quad \text{ implies }\quad  h(\mu)>\eps.$$
If there is no such value then $q_\phi^-:=-\infty$.
\end{list}

\begin{lema}
If $f\in \F$ and $\phi\in \P$ then $q_\phi^-\le 0$ and $q_\phi^+\ge 1$.
\label{lem:q_c range}
\end{lema}

\begin{proof}
Suppose $q\in (0, 1)$.  Then $T_\phi(q)\ge 0$.  Suppose that there is an equilibrium state $\mu_{\phi_q}\in \M$ for $\phi_q$.  Then by definition $$T_\phi(q)\lambda(\mu_{\phi_q})= h(\mu_{\phi_q})+q\int\phi~d\mu_{\phi_q}\ge 0$$
since by \cite{Prz}, $\lambda(\mu_{\phi_q})\ge 0$.  Since $q\int\phi~d\mu_{\phi_q}<0$, we must have $h(\mu_{\phi_q})>0$.
The lemma then follows by extending this argument to the case of measures $\mu$ with $h(\mu)+\int\phi_q~d\mu$ close to 0.
\end{proof}

For here on, we assume that $f\in \F_D$ to ensure that H\"older potentials $\phi$ yield induced potentials $\Phi$ for our inducing schemes which are locally H\"older continuous, as in \cite[Lemma 4]{BTeqgen}.

\begin{rem}
By \cite[Theorem 6]{BTeqgen}, if $f\in \F_D$ and $\phi\in \P_H$ and $\phi_{\max}-\phi_{\min}<\htop(f)$ then $q_\phi^+>1$.
\end{rem}

\begin{teo} \label{thm:unique mu_q}
Suppose that $f\in \F_D$ and $\phi\in \P_H$.  Then for every $q \in (q_\phi^-, q_\phi^+)$ the potential $\phi_q$ has a unique equilibrium measure $\mu_{\phi_q}$. Moreover, it is a measure of positive entropy.
\end{teo}

Since the proof of this theorem goes along the same lines as the proof of Theorem~\ref{thm:ITeq_exist_unique} given in \cite{IT}, we only sketch it here. Note that Theorem~\ref{thm:eq cts pots} is a corollary of this.

\begin{proof}
Proposition~\ref{prop:phase trans at 0} implies that there exists $\tilde q_\phi^-\in [-\infty, 0]$ such that for every $q\in (\tilde q_\phi^-, \infty)$ there exists a unique root $T(q) \in \R$ of the equation $P(-t \log |Df| +q \phi)=0$.

Lemma~\ref{lem:q_c range} implies that for $q\in (q_\phi^-, q_\phi^+)$ and any measure $\mu\in \M$ with $h(\mu)+\int\phi_q~d\mu$ close to 0 must have strictly positive entropy.

The rest of the proof follows as in \cite[Section 5]{IT}.  The steps are as follows:

\emph{Approximation of the pressure with compatible measures.}  The first step in the proof is to construct an inducing scheme, such that there exists a sequence of measures that approximate the pressure and are all compatible with it. More precisely:

\begin{prop}
Suppose that $f\in \F_D$ and $\phi\in \P_H$.  Let $q \in (q_\phi^-, q_\phi^+)$, then there exists an inducing scheme $(X,F)$ and a sequence of measures $(\mu_n)_n \subset \M$ all compatible with $(X,F)$ such that
\[ h(\mu_n) -T(q) \int \log |Df| \ d \mu_n +q \int \phi \ d \mu_n \to 0   \textrm{ and } \inf_n h(\mu_n) >0.                  \]
Moreover, if $\Phi_q$ denotes the induced potential of $\phi_q$ then $P(\Phi_q)=0$.
\end{prop}

The proof of this results follows from two observations: the first is that by definition there exist $\eps, \delta >0$ such that any measure $\mu$ with
$$\left|h(\mu)+\int\phi_q~d\mu\right|<\delta$$
is such that $h(\mu) >\eps$.
The other result used in the proof is that, given $\eps >0$ there exists a finite number of inducing schemes, such that any ergodic measure with $h(\mu) >\eps$ is compatible with one of these schemes and has integrable return time (this was first proved in \cite[Remark 6]{BTeqnat}, see also \cite[Lemma 4.1]{IT}).
Combining the previous two observations we obtain that $P(\Phi_q) \geq 0$. The fact that
$P(\Phi_q) \leq 0$ follows from an approximation argument (see \cite[Lemma 3.1]{IT}).  We note here that the potential $\Phi_q$ has summable variations by combining \cite[Lemma 4]{BTeqgen} and \cite[Lemma 8]{BTeqnat}.

Recall that  the inducing system $(X,F)$ can be coded by a full-shift on a countable alphabet, hence we have a Gibbs measure $\mu_{\Phi_q}$ corresponding to $\Phi_q$.

\emph{The Gibbs measure has integrable inducing time.}
The next step is to show that the inducing time is integrable with respect to the Gibbs measure $\mu_{\Phi_q}$.  This follows as in \cite[Proposition 5.2]{IT}.

\emph{Uniqueness of the equilibrium measure.}
This follows as in \cite[Proposition 6.1]{IT}.
\end{proof}

A detailed study of the temperature function will allow us to describe the multifractal spectrum. In order to study the regularity properties of the function $T_{\phi}(q)$ we need to understand the thermodynamic formalism for the potential $\phi_q$.

\begin{teo}  Suppose that $f\in \F_D$ and $\phi\in \P_H$.
If $q \in (q_\phi^-, q_\phi^+)$ then
\begin{enumerate}[a)]
\item the temperature function, $q \mapsto T_{\phi}(q)$ is differentiable;
\item
$DT_\phi(q) = \frac{\int \phi \ d \mu_{\phi_q}}{\int \log |Df| \ d \mu_{\phi_q}};$
\item $T_\phi(q)= \dim_H(\mu_{\phi_q})+q DT_\phi(q)$;
\item $T_{\phi}$ is convex;
\item if $f\in \F_{ac}$ and $\mu_{ac}\neq \mu_\phi$ then $T_{\phi}$ is strictly convex;
\item $T_\phi$ is linear in $(-\infty, q_\phi^-)$ and $(q_\phi^+, \infty)$;
\item  $T_\phi$ is $C^1$ at $q_\phi^+$.
\end{enumerate}
\label{thm:Tphi properties}
\end{teo}

\begin{proof}

\emph{Part a).}
It is a consequence of Theorem \ref{thm:unique mu_q} and \cite[Proposition 8.1]{IT}  that given $q \in (q_\phi^-, q_\phi ^+)$ there exists $\epsilon>0$ such that if $t \in (T_{\phi}(q) - \epsilon, T_{\phi}(q) + \epsilon)$ the pressure function
\[(t,q) \mapsto P(t,q)=P(-t \log |Df| +q \phi)\]
is differentiable in each variable. Therefore, by the implicit function theorem we obtain that $T_{\phi}(q)$ is differentiable.

\emph{Part b).}
This has been proved in several settings (see \cite[Proposition 21.2]{Pesbook}).
Consider the pressure function on two variables $$(t,q) \to P(t,q)=P(-t \log |Df| +q \phi).$$
There exists $\epsilon >0$ such that $P(t,q)$ is differentiable  on each variable in the range $t \in (T_\phi(q)-\epsilon, T_\phi(q)+ \epsilon)$ and $q \in \mathbb{R}$ satisfying the hypothesis  of the theorem.
As in for example \cite[Chapter 8]{PrzUrb_book}, \cite[Section II]{PesWei_mult} or \cite[Chapter 7 p.211]{Pesbook},
\[ DT_\phi(q) = \frac{\partial P(q,t)}{\partial t} \Big|_{t=T_\phi(q)} \left(  \frac{\partial P(q,t)}{\partial q}\Big|_{t=T_\phi(q)}  \right)^{-1}.\]
Furthermore, formulas for the derivative of the pressure (recall that it is differentiable in this range) give
\[DT_\phi(q) = \frac{\int \phi \ d \mu_{\phi_q}}{\int \log |Df| \ d \mu_{\phi_q}}\]
as required.  Note that in the above references the analogues of the remaining parts of the proof of this theorem would be proved using higher derivatives of $T_\phi$.  However, we do not have information on these, so we have to use other methods in the rest of this proof to get convexity etc.

\emph{Part c).}   Using b) and \cite{Hofdim},
\begin{align*}
T_\phi(q)&= \frac{h(\mu_{\phi_q})}{\int \log |Df| \ d\mu_{\phi_q}} +q \frac{\int \phi \ d\mu_{\phi_q}}{\log |Df| d \mu_{\phi_q}}=
\dim_H(\mu_{\phi_q})  +q \frac{\int \phi \ d\mu_{\phi_q}}{\log |Df| d \mu_{\phi_q}}\\
& = \dim_H(\mu_{\phi_q})  +q DT_\phi(q).
\end{align*}

\emph{Part d).}
Given $q\in (q_\phi^-, q_\phi^+)$ there is an equilibrium state $\mu_{\phi_q}$ for $\phi_q$.  We can write
$$T_\phi(q)=\frac{h(\mu_{\phi_q})+q\int\phi~d\mu_{\phi_q}}{\lambda(\mu_{\phi_q})}.$$
By the definitions of $T_\phi$ and pressure, for $\kappa\in \R$,
$$T_\phi(q+\kappa)\ge \frac{h(\mu_{\phi_q})+(q+\kappa)\int\phi~d\mu_{\phi_q}}{\lambda(\mu_{\phi_q})} = T_\phi(q)+\frac{\kappa\int\phi~d\mu_{\phi_q}}{\lambda(\mu_{\phi_q})} =T_\phi(q)+\kappa DT_\phi(q).$$
Whence $T_\phi$ is convex in $(q_\phi^-, q_\phi^+)$.

\emph{Part e).}
To show strict convexity, we use an improved version of the argument in \cite[Lemma 6]{T}.  There it is shown that if the graph of $T_\phi$ is not strictly convex then it must be affine. Similarly in this case suppose that $DT_\phi$ has slope $-\gamma$ in the interval $[q_1, q_2]\subset [q_\phi^-, q_\phi^+]$.  It can be derived from the above computations that the equilibrium state for $\phi_q$ is the same for all $q\in [q_1, q_2]$ (see also, for example, the proof of \cite[Lemma 6]{T}).

We will show that if $T_\phi$ is not strictly convex then $\mu_{\phi}$ is equivalent to the acip.  Let $(X,F)$ be an inducing scheme as in Theorem~\ref{thm:schemes} to which $\mu_{\phi_q}$ is compatible.
By the Gibbs property of $\mu_{\Phi_q}$ for $q, q+\delta\in [q_1, q_2]$, and for `$\asymp_{dis}$' meaning `equal up to a distortion constant' we must have
$$|X_i|^{T_\phi(q)}e^{q\Phi_{i}}\asymp_{dis} |X_i|^{T_\phi(q+\delta)}e^{(q+\delta)\Phi_{i}} = |X_i|^{T_\phi(q)-\delta\gamma}e^{(q+\delta)\Phi_{i}}$$
where $\Phi_{i}:=\sup_{x\in X_i}\Phi(x)$.   This implies $|X_i|^{\gamma} \asymp_{dis} e^{\Phi_{i}}$.  We can extend this argument from 1-cylinders to any $k$-cylinder.  This implies that we have a Gibbs measure $\mu_{\Phi/\gamma}$ for the potential $\Phi/\gamma$, and indeed that $\mu_{-\log|DF|} \equiv \mu_{\Phi/\gamma}$.  This also shows that $P(\Phi/\gamma)=0$.  Since $f$ has an acip, $\int\tau~d\mu_{\Phi/\gamma}<\infty$ and $\mu_{\Phi/\gamma}$ projects to a measure $\mu_{\phi/\gamma}$.  By Theorem~\ref{thm:schemes}(b), $\mu_{\phi/\gamma}$ must be an equilibrium state  for $\phi/\gamma$ as well as for $-\log|Df|$, i.e. $\mu_{\phi/\gamma}=\mu_{ac}$.  Moreover, $P(\phi/\gamma)=0$.  Since $\phi<0$,
$$\gamma>1 \text{ implies } P(\phi)<P(\phi/\gamma) \quad \text{ and } \quad \gamma<1 \text{ implies } P(\phi)>P(\phi/\gamma).$$
Since $P(\phi)=P(\phi/\gamma)=0$, we must have $\gamma=1$, so $\mu_{\phi}=\mu_{ac}$ contradicting our assumption.

\emph{Part f).}
We may assume that $q_\phi^-<0$.
Since the entropy of measures around $q_\phi^-$ is vanishingly small, we must have $$T_\phi(q)=\lim_{q\searrow q_\phi^-}\frac{q_\phi^-\int\phi~d\mu_{\phi_q}}{\lambda(\mu_{\phi_q})}.$$ If $T_\phi$ was not linear in $(-\infty, q_\phi^-)$, we must have measures $\mu$ with  $\frac{\int\phi~d\mu}{\lambda(\mu)}> \frac{\int\phi~d\mu_{\phi_q}}{\lambda(\mu_{\phi_q})}$.  This contradicts the definition of the value of $T_\phi(q_\phi^-)$.  A similar argument follows for $q_\phi^+$.

\emph{Part g).}
For  $q\ge q_\phi^+$, we have $T_\phi(q)<0$ and so $\phi_{q}$ is upper semicontinuous and there is an equilibrium state for $\phi_{q}$.
Using part f) we can show any equilibrium state $\mu_{\phi^+}$ for $\phi_q$ for $q\ge q_\phi^+$ is an equilibrium state for $\phi_q$ for any other $q\ge q_\phi^+$.  Since $-T_\phi(q)\lambda(\mu_{\phi^-})+q\phi=0$, and $\phi<0$, we have $\lambda(\mu_{\phi^-})>0$.   If $T_\phi$ was not $C^1$ at $q_\phi^-$ then we could take a limit $\mu$ of the measures $\mu_{\phi_q}$ where $q\to q_{\phi}^-$.
As in the proof of \cite[Theorem B]{IT}, $\mu$ must be an equilibrium state for $\phi_{q_\phi^-}$ with $\lambda(\mu)>0$, and not equal to $\mu_{\phi^-}$.  As in \cite[Proposition 6.1]{IT}, there can be at most one equilibrium state for $\phi_{q_\phi^-}$ of positive Lyapunov exponent.  Hence $T_\phi$ is $C^1$ at $q_\phi^-$, as required.
\end{proof}

\section{Multifractal spectrum of pointwise dimension} \label{sec:dspec proof}
In this section we prove that the dimension spectrum of pointwise dimension $\tilde \dimspec_{\mu}$  is the Legendre-Fenchel transform of the temperature function $T_\phi$.  The following is a slightly embellished version of Theorem~\ref{thm:main multi}.
\begin{teo}  \label{thm:mfd}
 Suppose that $f\in \F_D$ and $\phi\in \P_H$.  If $\mu_\phi\neq \mu_{ac}$ then the dimension spectrum satisfies the following equations
\begin{equation*}
 \dimspec_{\mu_\phi}(\alpha)= \inf_{q \in \mathbb{R}} \left(T_\phi(q) +q \alpha \right)
\end{equation*}
for all  $\alpha \in (-DT_\phi(q_\phi^+), -D^+T_\phi(q_\phi^-))$.  Or equivalently,
\[ \dimspec_{\mu_\phi}(-DT_\phi(q))= T_\phi(q) -q DT_\phi(q) \]
for $q\in (q_\phi^-, q_\phi^+]$.
\end{teo}

\begin{rem}
As in Remark~\ref{rem:lack of LE upper bd}, we expect that $\dimspec_\mu(\alpha)=\inf_{q \in \mathbb{R}} \left(T_\phi(q) +q \alpha \right)$ for $\alpha\in [-D^+T_\phi(q_\phi^-),-D^-T_\phi(q_\phi^-)]$.
Similarly, if $q\notin [q_\phi^-, q_\phi^+]$ then any equilibrium state $\mu$ for $\phi_q$  must have $h(\mu)=0$.  In this case $\dim_H(\mu)=0$.  This suggests that $\dimspec_\mu(\alpha)=0$ for $\alpha\notin (-DT_\phi(q_\phi^+),-D^-T_\phi(q_\phi^-)]$.
\label{rem:lack of dim upper bd}
\end{rem}

\begin{proof}[Proof of Theorem \ref{thm:mfd}, the lower bound]
Let us start considering the range of values for which the function $T_\phi(q)$ is finite. Denote by $\mu_{\phi_q}$ the unique equilibrium measure corresponding to
$-T_\phi(q) \log |Df| + q \phi$ from Theorem~\ref{thm:unique mu_q}.

As in the proof of \cite[Theorem A]{T}, $d_{\mu_{\phi_q}}(x)=\frac{-\int\phi~d\mu_{\phi_q}}{\lambda(\mu_{\phi_q})}$.  Therefore

\[\mu_{\phi_q} \left(I \Sm K \left( \frac{-\int \phi \ d \mu_{\phi_q}}{\int \log |Df| d \mu_{\phi_q}} \right) \right) = \mu_{\phi_q} \left(I\Sm \tilde K \left( \frac{-\int \phi \ d \mu_{\phi_q}}{\int \log |Df| d \mu_{\phi_q}} \right) \right)= 0.\]
Therefore by Theorem~\ref{thm:Tphi properties} c),
\begin{align*}\dimspec_{\mu_\phi} \left(K \left( \frac{-\int \phi \ d \mu_{\phi_q}}{\int \log |Df| d \mu_{\phi_q}} \right) \right)  
&\ge \dim_H(\mu_{\phi_q})=
T_\phi(q) -q DT_\phi(q).\end{align*}
In this way we obtain the lower bounds for $\dimspec_{\mu_\phi}$ 
for any $\alpha$ in the range of the derivative of $T_\phi$.
\end{proof}

\begin{proof}[Proof of Theorem \ref{thm:mfd}, the upper bound]
The proof of the upper bound follows exactly as in the proof of the upper bound in \cite[Theorem A]{T}.  We note that the H\"older assumption on $\phi$ ensures that for an inducing scheme $(X, F)$ the induced measure $\tilde\mu_\phi$ has $d_{\tilde\mu_\phi}(x)=d_{\mu_\phi}(x)$ for any $x\in X\cap (X,F)^\infty$.
\end{proof}


The following result is a consequence of the Legendre-Fenchel relation between the temperature function and the dimension spectrum. Let us stress that there is  strong contrast between the behaviour of the dimension spectrum described in Theorem \ref{thm:Dunbdd dom} and the dimension spectrum for equilibrium states in hyperbolic systems (see for example \cite[Chapter 7]{Pesbook}). The lack of hyperbolicity of the map $f$ is reflected in the regularity properties of the spectrum.

\begin{teo}  \label{thm:Dunbdd dom}
Suppose that $f\in \F_D , \phi\in \P_H$ and  $\mu_{ac}\neq \mu_\phi$. Assume that the temperature function is such that
\[
T_\phi(q)=
\begin{cases}
\text{infinite} & \text{ if } q < 0,\\
finite & \text{ if } q \ge 0.
\end{cases}
\]
Then the domain of $\dimspec_{\mu_\phi}$ is unbounded. Moreover,
$D^+T_{\phi}(0) =\frac{\int\phi~d\mu_{ac}}{\lambda(\mu_{ac})}$ and for every $\alpha \geq -D^+T_{\phi}(0)$ we have that $\dimspec_{\mu_\phi}(\alpha)= T_\phi(0)=1$.
\end{teo}

\begin{proof}
The usual derivative formulas imply that if there exists a measure $\mu_{\phi_0}$ for the potential $\phi_0$ then $D^+T_\phi(0)=\frac{\int\phi~d\mu_{\phi_0}}{\lambda(\mu_{\phi_0})}$.  Since $\phi_0:=-\log|Df|$, as in \cite{Led}, $\mu_{\phi_0}=\mu_{ac}$ the acip.  The fact that $\dimspec_{\mu_\phi}(\alpha)= T_\phi(0)$ for $\alpha \geq -D^+T_{\phi}(0)$ follows as in Lemma~\ref{lem:small LE big dim}.
\end{proof}

We finish this paper by giving a proposition which gives further information on the condition $\mu_\phi\neq\mu_{ac}$ imposed in the above theorems.  One way that $\mu_\phi$ can be equal to $\mu_{ac}$ is if $\phi$ is \emph{cohomologous} to $-\log|Df|$, that is if there exists a solution $\psi:I \to \R$ to the equation
\begin{equation}\phi=-\log|Df|+ \psi\circ f -\psi.\label{eq:cohom} \end{equation}
It is unknown if this is the only way that $\mu_\phi$ can be equal to $\mu_{ac}$.  The study of such equations, and their smoothness is part of Liv\v sic theory, studied for interval maps with critical points in \cite{BrHoNi}.

Let $\F_D'\supset\F_D$ be the class of maps as above, but allowing preperiodic critical points.  The following result is proved using ideas from \cite{BrHoNi}.

\begin{prop}\label{prop:cohom}
Let $f\in\F_D'$ be a unimodal map.  If $\phi:I \mapsto \R$ is a H\"older function then
the only way \eqref{eq:cohom} can have a solution  is if the critical point is preperiodic.
\end{prop}

\begin{proof}
Theorem 6 of \cite{BrHoNi} holds for $f\in \F_D'$.  Note that we require weaker conditions on the growth of derivatives using the condition in \cite[Proposition 6]{T}.  Therefore the potential $\phi':=\phi+\log|Df|$ satisfies the conditions in \cite[Theorem 6]{BrHoNi}: in particular it satisfies condition (2) of \cite[Section 3.1]{BrHoNi} for example.  By that theorem, any solution $\psi$ to the equation $\phi'=\psi\circ f -\psi$ must be H\"older continuous.
Letting $c$ be the critical point, we may assume that $f(c)$ is a maximum for $f$.  As in \cite[Corollary 3]{BrHoNi}, $\psi$ must be bounded on any interval compactly contained in $[f^2(c), f(c)]$.  But by construction, $\psi$ must be unbounded on any element of $\cup_{n\ge 1}f^n(c)$.  In the case of transitive unimodal maps, this can only occur when $f^2(c)=0$ and $0$ is a fixed point.
\end{proof}


\begin{thebibliography}{99}

\bibitem[AL]{AvLyu} A.\ Avila, M.\  Lyubich,
\emph{Hausdorff dimension and conformal measures of Feigenbaum Julia sets,}
 J. Amer. Math. Soc. \textbf{ 21}  (2008) 305--363.

\bibitem[Ba]{Bar} L.\ Barreira, \emph{Dimension and recurrence in hyperbolic dynamics,} Progress in Mathematics, 272. Birkh\"auser Verlag, Basel, 2008.

\bibitem[BaI]{BarIo} L.\ Barreira G.\ Iommi,
 \emph{Phase transitions and Multifractal analysis for parabolic horseshoes,} to appear in Israel J. Math.

\bibitem[BPS]{BarPeSc}  L.\ Barreira, Y.\ Pesin, J.\ Schmeling,
 \emph{On a general concept of multifractality: multifractal spectra for dimensions, entropies, and Lyapunov exponents. Multifractal rigidity,} Chaos
 \textbf{7} (1997) 27--38.

\bibitem[BaS]{BarSc} L.\ Barreira, J.\ Schmeling,
  \emph{Sets of ``non-typical'' points have full topological entropy and full Hausdorff dimension}, Israel J. Math. \textbf{116} (2000) 29--70.






\bibitem[BHN]{BrHoNi} H.\ Bruin, M.\ Holland, M.\ Nicol,
\emph{Liv\v{s}ic regularity for Markov systems,}
Ergodic Theory Dynam. Systems \textbf{25} (2005) 1739--1765.

\bibitem[BK]{BrKell} H.\ Bruin, G.\ Keller,  \emph{Equilibrium states for $S$-unimodal maps,} Ergodic Theory Dynam. Systems \textbf{18} (1998) 765--789.




\bibitem[BRSS]{BRSS} H.\ Bruin, J.\ Rivera-Letelier, W.\ Shen, S.\ van Strien,
{\em Large derivatives, backward contraction and invariant densities
for interval maps,} Invent. Math. {\bf 172} (2008) 509--593.

\bibitem[BT1]{BTeqgen} H.\ Bruin, M.\ Todd,
 \emph{Equilibrium states for potentials with $\sup \phi-\inf \phi < h_{top}(f)$, }  Comm. Math. Phys. \textbf{ 283} (2008) 579-611.

\bibitem[BT2]{BTeqnat} H.\ Bruin, M.\ Todd,
\emph{Equilibrium states for interval maps: the potential $ -t \log|Df|$,}
Ann. Sci. \'Ecole Norm. Sup. (4) \textbf{ 42} (2009) 559--600.


\bibitem[C]{Ced} S.\ Cedervall,
\emph{Invariant measures and decay of correlations for S-multimodal interval maps,}
Thesis, Imperial College, London, 2006.






\bibitem[De]{DenUrb} M.\ Denker, M.\ Urba\'nski,
\emph{On the existence of conformal measures,}
Trans. Amer. Math. Soc. \textbf{328} (1991) 563--587.

\bibitem[D1]{Dobcusp} N.\ Dobbs,
\emph{On cusps and flat tops,}
Preprint,  arXiv:0801.3815.

\bibitem[D2]{Dob_comp} N.\ Dobbs,
\emph{Measures with positive Lyapunov exponent and conformal measures in rational dynamics,}
Preprint, arXiv:0804.3753.

\bibitem[GPR]{GelPrRa} K.\ Gelfert, F.\ Przytycki, M.\ Rams,
 \emph{Lyapunov spectrum for rational maps,}  Preprint,  arXiv:0809.3363.


\bibitem[GR]{GelRa} K.\ Gelfert, M.\ Rams,  \emph{The Lyapunov spectrum of some parabolic systems,} Ergodic Theory Dynam. Systems {\bf 29} (2009) 919-940.





\bibitem[HMU]{HaMauUr} P.\ Hanus, R.D.\ Mauldin, M.\ Urba\'nski , \emph{Thermodynamic formalism and mutifractal analysis of conformal infinite iterated function systems,} Acta Math. Hungar. {\bf 96} (2002) 27-98.

\bibitem[Ho]{Hofdim} F.\ Hofbauer,
{\em Local dimension for piecewise monotonic maps on the interval,}
Ergodic Theory Dynam. Systems {\bf 15} (1995) 1119--1142.

\bibitem[HRS]{HoRaSt} F.\ Hofbauer, P.\ Raith, T.\ Steinberger,
{\em Multifractal dimensions for invariant subsets of piecewise monotonic interval maps,}
Fund. Math. {\bf 176}  (2003) 209--232.

\bibitem[I1]{io1} G. Iommi, \emph{Multifractal analysis for countable Markov shifts,} Ergodic Theory Dynam. Systems \textbf{25} (2005) 1881-1907.

\bibitem[I2]{Iom} G.\ Iommi, \emph{Multifractal analysis of Lyapunov exponent for the backward continued fraction map,}  To appear in Ergodic Theory Dynam. Systems.

\bibitem[IK]{IomKi} G.\ Iommi, J.\ Kiwi, \emph{The Lyapunov spectrum is not always concave,}  J. Stat. Phys. \textbf{135} 535-546 (2009).

\bibitem[IT]{IT} G.\ Iommi, M.\ Todd,
 \emph{Thermodynamic formalism for multimodal maps,} Preprint, arXiv:0907.2406.

\bibitem[JR]{JoRa} T.\ Jordan, M.\ Rams, \emph{Multifractal analysis of weak Gibbs measures for non-uniformly expanding $C^1$ maps,} To appear in Ergodic Theory Dynam. Systems, arXiv:0806.0727.


\bibitem[K]{Kellbook} G.\ Keller,
{\em Equilibrium states in ergodic theory,} London Mathematical Society Student Texts, 42. Cambridge University Press, Cambridge, 1998.

\bibitem[KeS]{KeS}  M.\ Kesseb\"ohmer, B.\ Stratmann  \emph{A multifractal analysis for Stern-Brocot intervals, continued fractions and Diophantine growth rates,}  Journal f\"ur die reine und angewandte Mathematik (Crelles Journal) \textbf{605} (2007) 133-163.

\bibitem[L]{Led} F. Ledrappier, \emph{ Some properties of absolutely continuous invariant measures on an interval,} Ergodic Theory Dynam. Systems \textbf{1} (1981) 77--93.


\bibitem[MU]{muGIBBS}
R.\ Mauldin, M.\ Urba\'nski, \emph{Gibbs states on the symbolic
space over an infinite alphabet}, Israel J. Math. \textbf{125}
(2001), 93--130.

\bibitem[MS]{MSbook} W.\ de Melo, S.\ van Strien,
{\em One dimensional dynamics,} Ergebnisse Series {\bf 25}, Springer--Verlag, 1993.

\bibitem[Na]{Nak} K.\ Nakaishi, \emph{Multifractal formalism for some parabolic maps,} Ergodic Theory Dynam. Systems \textbf{24} (2000) 843-857.


\bibitem[NS]{NoSa} T.\ Nowicki, D.\ Sands,
{\em Non-uniform hyperbolicity and universal bounds for $S$-unimodal maps,}
Invent. Math. {\bf 132}  (1998) 633--680.

\bibitem[O]{Ols} L.\ Olsen, \emph{A multifractal formalism,} Advances in Mathematics \textbf{116} (1995) 82-196.

\bibitem[P]{Pesbook} Y.\ Pesin,  \emph{Dimension Theory in Dynamical Systems,} CUP 1997.


\bibitem[PW]{PesWei_mult} Y.\ Pesin, H.\ Weiss,
\emph{A multifractal analysis of equilibrium measures for conformal expanding maps and Moran-like geometric constructions,}
J. Statist. Phys. {\bf 86} (1997) 233--275.



\bibitem[PolW]{PolWe} M.\ Pollicott, H.\ Weiss, \emph{Multifractal analysis of Lyapunov exponent for continued fraction and Manneville-Pomeau transformations and applications to Diophantine approximation,}  Comm. Math. Phys. \textbf{ 207} (1999) 145--171.



\bibitem[Pr]{Prz} F.\ Przytycki, {\em Lyapunov characteristic exponents are nonnegative,} Proc. Amer. Math. Soc. {\bf 119} (1993) 309--317.


\bibitem[PrU]{PrzUrb_book} F.\ Przytycki, M. Urba\'nski,
\emph{Fractals in the Plane, Ergodic Theory Methods,}
to appear in Cambridge University Press. Available at http://www.math.unt.edu/~urbanski/book1.html.



\bibitem[RS]{rivshen}  J.\ Rivera-Letelier, W.\ Shen
{\em On statistical properties of one-dimensional maps with weak hyperbolicity.} Article in preparation.

\bibitem[S1]{Sartherm} O.\ Sarig, {\em Thermodynamic formalism for
    countable Markov shifts,}  Ergodic Theory Dynam. Systems {\bf
    19}  (1999) 1565--1593.

\bibitem[S2]{Sarphase} O.\ Sarig,  \emph{Phase transitions for countable Markov shifts,}  Comm. Math. Phys.  \textbf{217}  (2001) 555--577.

\bibitem[S3]{SarBIP} O.\ Sarig, \emph{Existence of Gibbs measures for countable Markov shifts,} Proc. Amer. Math. Soc. \textbf{131} (2003) 1751--1758.


\bibitem[Sch]{sch} J.\ Schmeling, \emph{On the completeness of multifractal spectra,} Ergodic Theory Dynam. Systems \textbf{19} (1999) 1595--1616.

\bibitem[SV]{SVarg} S.\ van Strien, E.\ Vargas,
{\em Real bounds, ergodicity and negative Schwarzian for
multimodal maps,}  J. Amer. Math. Soc.  {\bf 17}  (2004) 749--782.

\bibitem[T]{T} M.\ Todd,  \emph{Multifractal analysis for multimodal maps,} Preprint 2008, arXiv:0809.1074.


\bibitem[Wa]{Walbook}
P.\ Walters, \emph{An Introduction to Ergodic Theory}, Graduate Texts
in Mathematics 79, Springer, 1981.

\bibitem[W]{Wei} H.\ Weiss, \emph{The Lyapunov spectrum for conformal expanding maps and Axiom A surface diffeomorphisms,} J. Statist. Phys. \textbf{95} (1999) 615-632.


\end{thebibliography}
\end{document}